\documentclass[11pt]{amsart}
\usepackage{amsfonts,amsmath,amssymb,latexsym}
\usepackage{eucal}
\usepackage{geometry}
\usepackage{graphics}
\usepackage{color}

\newtheorem{proposition}{Proposition}
\newtheorem{theorem}{Theorem}
\newtheorem{claim}{Claim}
\newtheorem{lemma}{Lemma}
\newtheorem*{theoremmain}{Theorem~\ref{main}}

\theoremstyle{definition}
\newtheorem{definition}{Definition}

\begin{document}

\title{A half-space theorem for graphs of constant mean curvature
$0<H<\frac{1}{2}$ in $\mathbb{H}^2\times\mathbb{R}$}

\author{L. Mazet}
\address{Universit\'e Paris-Est, Laboratoire d'Analyse et de Math\'ematiques
Appliqu\'ees (UMR 8050), UPEC, UPEM, CNRS, F-94010, Cr\'eteil, France}
\email{laurent.mazet@math.cnrs.fr}

\author{G. A. Wanderley}
\address{Universit\'e Paris-Est, Laboratoire d'Analyse et de Math\'ematiques
Appliqu\'ees (UMR 8050), UPEM, UPEC, CNRS, F-94010, Cr\'eteil, France}
\email{gaw.gabriela@googlemail.com}

\thanks{The authors were partially supported by the ANR-11-IS01-0002 grant}

\begin{abstract}
We study a half-space problem related to graphs in
$\mathbb{H}^2\times\mathbb{R}$, where $\mathbb{H}^2$ is the hyperbolic plane,
having constant mean curvature $H$ defined over unbounded domains in
$\mathbb{H}^2$.
\end{abstract}

\maketitle

\section{Introduction}

The half-space theorem by Hoffman and Meeks \cite{HoffmanMeeks} states that if a
properly immersed minimal surface $S$ in $\mathbb{R}^3$ lies on one side of some
plane $P$, then $S$ is a plane parallel to $P$. As a consequence, they proved
the strong half-space theorem which says that two properly immersed minimal
surfaces in $\mathbb{R}^3$ that do not intersect must be parallel planes.

These theorems have been generalized to some other ambient simply connected
homogeneous manifolds with dimension 3. For example, we have half-space theorems
with respect to horospheres in $\mathbb{H}^3$ \cite{RodriRosen}, vertical minimal planes
in $Nil_3$ and $Sol_3$ \cite{DaniHaus},\cite{DaniMeeksRosen} and entire minimal
graph in $Nil_3$ \cite{DaniMeeksRosen}.
 It is known that there is no half-space theorem for horizontal slices in
 $\mathbb{H}^2\times\mathbb{R}$, since rotational minimal surfaces (catenoids)
 are contained in a slab \cite{Nelli1}-\cite{Nelli2}, but one has half-space
 theorems for constant mean curvature $\frac{1}{2}$ surfaces in
 $\mathbb{H}^2\times\mathbb{R}$ \cite{HausRosenSpru}.

In \cite{Mazet}, the first author proved a general half-space theorem for constant mean
curvature surfaces. Under certain hypothesis, he proved that in a Riemannian
$3$-manifold of bounded geometry, a constant mean curvature $H$ surface on one
side of a parabolic constant mean curvature $H$ surface $\Sigma$ is an
equidistant surface to $\Sigma$.

In Euclidian spaces of dimension higher than $4$, there is no strong half-space
theorem, since there exist rotational proper minimal hypersurfaces contained in
a slab.

In \cite{Menezes}, Menezes proves a half-space theorem for some complete vertical
minimal graphs, more precisely, she looks at some particular graphs
$\Sigma\subset M\times\mathbb{R}$ over a domain $D\subset M$, where $M$ is a
Hadamard surface with bounded curvature, these graphs are called ideal Scherk graphs and
their existence was proved by Collin and Rosenberg in \cite{CollRose} for $\mathbb{H}^2$
and by Galvez and Rosenberg in \cite{GalvRosen} in the general case.

\begin{theorem}\label{ana}
Let $M$ denote a Hadamard surface with bounded curvature and let $\Sigma =
Graph(u)$ be an ideal Scherk graph over an admissible polygonal domain $D\subset
M$. If $S$ is a properly immersed minimal surface contained in
$D\times\mathbb{R}$ and disjoint from $\Sigma$, then $S$ is a translate of
$\Sigma$.
\end{theorem}

In this paper, we are interested in the case where the graph has constant mean curvature.
More precisely, we consider graphs over $\mathbb{H}^2$ with constant mean curvature
$0<H<\frac12$. In that case, we prove a result similar to the one of Menezes. We notice
that the value $H=\frac12$ is critical in this setting (see \cite{Mazet2} for the
$H=\frac12$ case).


The graphs that we will work with are graphs of functions $u$ defined in a
domain $D\subset \mathbb{H}^2$ whose boundary $\partial D$ is composed of complete arcs $\{A_i\}$
and $\{B_j\}$, such that the curvatures of the arcs with respect to the domain
are $\kappa (A_i)=2H$ and $\kappa (B_j)=-2H$. These graphs will have constant
mean curvature and $u$ will assume the value $+\infty$ on each $A_i$ and $-\infty$
on each $B_j$. These domains $D$ will be called \textit{Scherk type domains} and
the functions $u$ \textit{Scherk type solutions}. The existence of these graphs
is assured by A. Folha and S. Melo in \cite{FolhaMelo}. There, the authors give
necessary and sufficient conditions on the geometry of the domain $D$ to prove
the existence of such a solution. In this context, we prove the following result.

\begin{theorem}\label{main}
Let $D\subset \mathbb{H}^2$ be a Scherk type domain and $u$ be a Scherk type
solution over $D$ (for some value $0<H<\frac12$). Denote by $\Sigma= Graph (u)$. If $S$ is
a properly immersed CMC $H$ surface contained in $D\times\mathbb{R}$ and above $\Sigma$,
then $S$ is a vertical translate of $\Sigma$.
\end{theorem}

The original idea of Hoffman and Meeks is to use the $1$-parameter family of catenoids as
\textit{a priori} barriers to control minimal surfaces on one side of a plane (here
\textit{a priori} means that the choice of catenoids is independent of the particular
minimal surface you want to control). In more
general situations, it is not easy to construct such a continuous family of barriers so
some authors use a discrete family (see for example \cite{DaniMeeksRosen,RosenSchulzSpruc}).
Menezes works also
with such a discrete family. In our case, it does not seem possible to construct such a
family in an easy way. Our approach is based on the existence of only one barrier whose
construction depends on the particular surface $S$.


This paper is organized as follows. In Section 2, we will give a brief
presentation of the Scherk type graphs and the result of Folha and Melo. Section 3
contains the proof of Theorem~\ref{main}, so one of the main step is the existence of the
barriers which uses the Perron method. We also prove a uniqueness result for the constant
mean curvature equation. 

\section{Constant mean curvature Scherk type graphs} 

In this section we present the theorem by A. Folha and S. Melo in
\cite{FolhaMelo} that assure the existence of constant mean curvature graphs
which take the boundary value $+\infty$ on certain arcs $A_i$ and $-\infty$ on
arcs $B_j$. All along this section $H$ will be a real constant in $(0,\frac12)$.

First, let us fix some notations. Let $\mathbb{H}^2$ be the hyperbolic plane, and
$\mathbb{H}^2\times\mathbb{R}$ be endowed with the product metric. Let $D$ be a simply
connected domain in $\mathbb{H}^2$ and $u:D\longrightarrow \mathbb{R}$ a
function. Denote by 
$$\Sigma=Graph (u)=\{(x,u(x)), x\in D\}.$$

The upward unit normal to $\Sigma$ is given by

\begin{equation}
N = \frac{1}{W}\big(\partial_t -  \nabla u\big),
\end{equation}

where

\begin{equation}
W = \sqrt{1 +|\nabla u|^2}.
\end{equation}

The graph $\Sigma$ has mean curvature $H$ if $u$ satisfies the equation

\begin{equation}
\label{linear}
\mathcal{L}u:= \textrm{div} \frac{\nabla u}{W}-2H=0,
\end{equation}

where the divergence and the gradient are taken with respect to the metric on
$\mathbb{H}^2$. Let us now give some definitions.

\begin{definition}
We say that the boundary of a domain $D$ in $\mathbb{H}^2$ is a \emph{$2H$-polygon}
if its boundary is made of a finite number of complete arcs with constant curvature $2H$
and the cluster points of $D$ in $\partial_\infty\mathbb{H}^2$ are the end-points of
these arcs. The arcs are called the edges of $D$ and the cluster points are the
vertices of $D$.
\end{definition}

We notice that a complete curve in $\mathbb{H}^2$ with constant curvature $2H$ is proper. 

If $\Omega$ is a domain whose boundary is a $2H$-polygon, we will denote by $A_i$ (resp.
$B_i)$ the arcs of the boundary whose curvature is $2H$ (resp. $-2H$) with respect to the
inward pointing unit normal.



\begin{definition} We say that a domain $D$ in $\mathbb{H}^2$ is a \emph{Scherk type
domain} if its boundary is a $2H$-polygon and if each vertex is the end point of one arc
$A_i$ and one arc $B_j$.
\end{definition}

Such a domain $D$ is drawn in Figure~\ref{fig}.

\begin{definition} Let $\Omega$ be a Scherk type domain. We say that $P$ is an
\emph{admissible inscribed polygon} if $P\subset \Omega$ is a domain whose boundary is a
$2H$-polygon and its vertices are among the ones of $\Omega$.
\end{definition}

Let $D$ be a Scherk type domain, in \cite{FolhaMelo}, Folha and Melo study the following
Dirichlet problem
\begin{equation}\label{dirichlet}
\begin{cases}
\mathcal{L}(u)=0 \textrm{ in }D\\
u=+\infty \textrm{ on }A_i\\
u=-\infty \textrm{ on }B_i
\end{cases}
\end{equation}

In order to state the result of Folha and Melo, let us introduce some notations.
Let $P$ be an admissible inscribed polygon in $D$ and let $\{d_i \}_{i\in I}$ denote the
vertices of $P$. Consider the set
$$
\Theta = \{(\mathcal{H}_i)_{i\in I} |\, \mathcal{H}_i\textrm{ is a horodisk at }d_i
\textrm{ and }\mathcal{H}_i \cap \mathcal{H}_j = \emptyset \textrm{ if } i \neq j \}.
$$
we notice that, by choosing sufficiently small horodisks, $\Theta$ is not empty.

Let $(\mathcal{H}_i)_{i\in I}$ be in $\Theta$ such that the following is true : each arc
$A_i$ and $B_j$ meets exactly two of these horodisks. Denote by $\tilde{A}_i$ the compact
arc of $A_i$ which is the part of $A_i$ outside these two horodisks. Let $|A_i |$ denote
the length of
$\tilde{A}_i$. We introduce the same notations of the $B_j$. For each arc $\eta_j\in
\partial P$, we also define $\tilde{\eta}_j$ and $|\eta_j |$ in the same way.

We define
$$
\alpha(\partial P)=\sum_{A_i\in \partial P} |A_i|, \quad \beta(\partial P)=\sum_{B_i\in \partial P} |B_i| \quad
\textrm{ and }\quad l(\partial P)=\sum_{j} |\eta_j|, 
$$

where $\partial P=\cup_j\eta_j$. We remark that a Scherk type domain has finite area. So we can introduce $\mathcal{A}(D)$ the area of $D$ and $\mathcal{A}(P)$ the area of $P$.






With these definitions we can state the main theorem of \cite{FolhaMelo}.

\begin{theorem}\label{abgail}
Let $D$ be a Scherk type domain. Then there exists a
solution $u$ for the Dirichlet problem \eqref{dirichlet} in $D$ if and only if for some
choice of the horodisks (in $\Theta$) at the vertices,
$$
\alpha(\partial D)=\beta(\partial D)+2H{\mathcal{A}}(D)
$$
and for any admissible inscribed polygons $P\neq D$,

$$
2\alpha(\partial P)<l(\partial P)+2H{\mathcal{A}}(\Omega) \textrm{ and }
2\beta(\partial P)<l(\partial P)-2H{\mathcal{A}}(\Omega).
$$
\end{theorem}

It could seem that the conditions depend on the choice of the horodisks in $\Theta$,
actually they are independent of that choice if the horodisks are small enough. The
details and the proof of this theorem can be found in \cite{FolhaMelo}.

\section{The main Result}

In this section we will prove the following result

\begin{theoremmain}
Let $D\subset \mathbb{H}^2$ be a Scherk type domain and $u$ be a Scherk type
solution over $D$ (for some value $0<H<\frac12$). Denote by $\Sigma= Graph (u)$. If $S$ is
a properly immersed CMC $H$ surface contained in $D\times\mathbb{R}$ and above $\Sigma$,
then $S$ is a vertical translate of $\Sigma$.
\end{theoremmain}

The proof of the theorem consists in constructing barriers to control the
surface $S$. Before starting the proof, let us give some notations and preliminary results
that we will use.

So we fix a value of $H\in(0,\frac12)$, a Scherk type domain $D$ and a Scherk type solution $u$. Let $y\in D$ and $B_y$ and $B'_y$
be open balls centered in $y$ such that $B_y\subsetneqq B'_y\subsetneqq D$ (see
Figure~\ref{fig}). The following result consists in constructing a first barrier to
control $S$.

\begin{figure}[h]
\begin{center}
\resizebox{0.5\linewidth}{!}{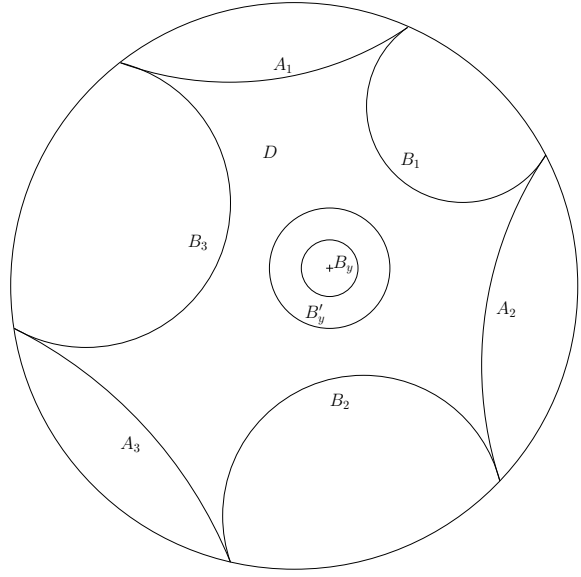}
\caption{The Scherk type domain $D$ and the balls $B_y$ and $B_y'$
\label{fig}}
\end{center}
\end{figure}

\begin{lemma}\label{lemma1} There exists a constant $\epsilon > 0$ such that for
all $t\in [0,\epsilon)$ there exists $v\in C^{2}(B'_y\setminus B_y)\cap
C^0(\partial(B_y'\setminus B_y))$ such that $v$ solves \eqref{linear} and $v=u$ in
$\partial B_y'$ and $v=u+t$ over $\partial B_y$.
\end{lemma}

\begin{proof}
Consider the operator $F:C^{2,\alpha}(\overline{B'_{y}\backslash
B_y})\times C^{2,\alpha}(\partial (B'_{y}\backslash B_y))\longrightarrow
C^{0,\alpha}(\overline{B'_{y}\backslash B_y})\times C^{2,\alpha}(\partial
(B'_{y}\backslash B_y))$ given by

$$F(v,\phi)=(\mathcal{L}v,v-\phi).$$

Observe that 
$$
F(u,u)=0.
$$

Moreover, consider the operator

\begin{align*}
T:= D_1 F(u,u): C^{2,\alpha}(\overline{B'_{y}\setminus B_y})\longrightarrow&
C^{0,\alpha}(\overline{B'_{y}\setminus B_y})\times
C^{2,\alpha}(\partial(B'_{y}\setminus B_y))\\
h\longmapsto& \lim_{t\longrightarrow 0}\frac{F(u+th,u)-F(u,u)}{t}
\end{align*}

We have that
$$
T(h)=\left(Div\left(\frac{\nabla h-\frac{\nabla u}{W}\langle\frac{\nabla
u}{W},\nabla h\rangle}{W}\right), h\right).
$$

Observe that $T$ is a linear operator, of the form $T=(T_1,T_2)$ where $T_1$ is
an elliptic operator of the form 
$$
T_1(v)=a^{ij}(x)D_{ij}v+b^i(x)D_i v;\quad a^{ij}=a^{ji}.
$$

Moreover, since $|\nabla u|\leq C$, we have that $\frac{|\nabla u|}{W}\leq
C'<1$, this implies that $T_1$ is uniformly elliptic. We also have that the
coefficients of $T_1$ belong to $C^{0,\alpha}(\overline{B'_y\setminus B_y})$. It
follows by Theorem 6.14 in \cite{GT} that if $g\in
C^{0,\alpha}(\overline{B'_y\setminus B_y})$ and $\phi \in C^{2,\alpha}(\partial
(B'_y\setminus B_y))$, then there exists a unique $w\in
C^{2,\alpha}(\overline{B'_y\setminus B_y})$ such that $T_1(w)=g$ in
$B'_y\setminus B_y$ and $w=\phi$ in $\partial (B'_y\setminus B_y)$. Then, $T_1$
is invertible.

We conclude that $T$ is invertible. It follows by the implicit function theorem
that for all $\phi$ close to $u$ there exists a solution of $\mathcal{L}v=0$ in
$B_y'\setminus \overline{B_y}$ with $v=\phi$ in $\partial(B_y'\setminus B_y)$.
In other words, it exists $\epsilon > 0$ such that for all $t\in [0,\epsilon)$
there exists $v$ such that $v$ solves \eqref{linear} and $v=u$ over $\partial
B_y'$ and $v=u+t$ over $\partial B_y$.
\end{proof}

Let $S$ be as in Theorem \ref{main}. Give $p\in D$, define $g(p)$ by
$$g(p)=\inf_{t\in\mathbb{R}}\{(p,t)\in S\}\in\mathbb{R}\cup\{+\infty\}.$$

Observe that $g$ is a lower semicontinuous functions and $g\geq u$. From now
on, we will assume that $g>u$ (the case where $g(p)=u(p)$ for a point $p\in
D$ will be considered in the proof of the theorem). Then
for $\epsilon'>0$ sufficiently small, we have that
\begin{equation}
\label{condg}g>u+\epsilon' \quad\textrm{ on }\partial B_y.
\end{equation}

Now, let $\epsilon$ be as in Lemma \ref{lemma1}, fix
$\epsilon'<\epsilon$ where $\epsilon'$ satisfies \eqref{condg} and $v$ given by
Lemma~\ref{lemma1} associated to $\epsilon'$. We will
construct a second barrier to control the surface $S$. More precisely, we will
prove the existence of a function $\beta$ that satisfies

\begin{align}
\label{eq1}
\mathcal{L}\beta &= 0 \quad\quad\textrm{ in } D\backslash \overline{B_y},\\
\label{eq2}
\beta&=  u +\epsilon' \quad\textrm{ in } \partial B_y.
\end{align}

\begin{proposition}\label{prop1} There is a solution $\beta\in C^2(D\setminus
\overline{B_y})\cap C^0(D\setminus B_y)$ for the Dirichlet problem
\eqref{eq1}-\eqref{eq2} such that $\max (u,v)\leq \beta\leq \min (u+\epsilon',
g)$.
\end{proposition}

\begin{proof} To prove this proposition we will use the Perron method. Let us
recall the framework of this method. A function $w\in C^0(D\setminus B_y)$ is
called a \textit{subsolution} for $\mathcal{L}$ if for any compact subdomain
$U\subset D\backslash B_y$ and any solution $h$ of \eqref{linear} with $w\leq h$
on the boundary of $U$, we have $w\leq h$ on U.

First, observe that $u$ is a subsolution for \eqref{linear}. Moreover, if $w'$ and
$w$ are subsolutions, the continuous function $\max (w',w)$ is also a
subsolution.

Let $\Delta\subset D\setminus B_y$ be a geodesic disk of small radius such that
$\kappa(\partial \Delta)\geq 2H$. Theorem 3.2 in \cite{HausRosenSpru2} implies
that the Dirichlet problem for Equation \eqref{linear} can be solved in $\Delta$.
So, for any such disk $\Delta$ and subsolution $w$, we can define a continuous
function $M_{\Delta}(w)$ as

$$
M_{\Delta}(w)(x)=\begin{cases}
w(x) \quad\textrm{if } x\in D\setminus\Delta \\
\nu(x) \quad\textrm{if } x\in \Delta 
\end{cases}
$$
where $\nu$ is the solution of $\mathcal{L}\nu = 0$ in $\Delta$, with $\nu=w$ in
$\partial\Delta$.

Also, define $u_+=\min (u+\epsilon', g)$ and $u_-=\max(u,v)$. Denote by $\Gamma$
the set of all subsolutions $w$ such that $w\leq u_+$ on $D\setminus B_y$,

\begin{claim}\label{claim1}
If $w\in \Gamma$,then $M_{\Delta}(w)\in \Gamma$.
\end{claim}
\begin{proof}[Proof of Claim~\ref{claim1}] First we have to prove that
$M_{\Delta}(w)$ is a
subsolution. So, take a arbitrary compact subdomain $U\subset D\setminus B_y$,
and let $h$ a solution of \eqref{linear} in $U$ with $M_{\Delta}(w)\leq h$ on
$\partial U$. Since $w=M_{\Delta}(w)$ in $U\setminus \Delta$ we have that
$M_{\Delta}(w)\leq h$ in $U\setminus\Delta$.

Moreover, $M_{\Delta}(w)$ is a solution of \eqref{linear} in $\Delta$. Then, by
the maximum principle, we have that $M_{\Delta}(w)\leq h$ in $U\cap \Delta$. So,
$M_{\Delta}(w)\leq h$ in $U$. Since $U$ is arbitrary, it follows that
$M_{\Delta}(w)$ is a subsolution in $D\setminus B_y$.

Now we have to prove that $M_{\Delta}(w)\leq u_+$. Observe that in
$D\setminus\Delta$, $M_{\Delta}(w)=w$, since $w\in\Gamma$, then $w\leq u_+$, and
so, $M_{\Delta}(w)\leq u_+$ in $D\setminus\Delta$.

On the other hand, $M_{\Delta}(w)=\nu$ in $\Delta$, where $\nu$ is a solution of
$\mathcal{L}(\nu)=0$ in $\Delta$ and $\nu=w$ on $\partial\Delta$. Thus $\nu\leq
u_+=\min(u+\epsilon',g)$ on $\partial\Delta$. It follows by the maximum
principle that $\nu\leq u+\epsilon'$ in $\Delta$. So, we have to prove that
$\nu\leq g$ in $\Delta$.

Suppose that there exists $q\in\Delta$ such that $(\nu-g)(q)<0$. Then, there
exists $p\in\Delta$ such that $(\nu-g)(p)=\min(\nu-g)=C<0$.

Now, observe that the graph of $g$ in $\Delta$ is a piece of the surface $S$,
let us denote it by $S_g$. Since $g\geq \nu+C$, then the graph $\Sigma_{\nu+C}$
of $\nu+C$ is a CMC $2H$ surface which is bellow the surface $S_g$. Moreover,
$(p,g(p))$ is a point of contact of $\Sigma_{\nu+C}$ and $S_g$, and by the
maximum principle, $S_g=\Sigma_{\nu+C}$. It follows that $g=\nu+C$ in $\Delta$,
since $\nu\leq g$ on $\partial\Delta$ then $C\geq0$, and this contradicts $C<0$.
Then $\nu\leq g$ in $\Delta$.
\end{proof}

We define our solution by the following formula

$$\beta(q)=\sup_{w\in \Gamma} w(q).$$

Observe that $u_-$ is a subsolution, since $u$ and $v$ are subsolutions. Also,
$u\leq u_+=\min(u+\epsilon',g)$. Moreover, in the proof of Claim \ref{claim1} we
see that $v\leq u_+$. Then $u_-=\max(u,v)\leq u_+$. We conclude that $u_-\in
\Gamma$, then $\Gamma$ is non empty, and $u_+$ is an upper bound for any $w$ in
$\Gamma$, thus $\beta$ is well defined.

The method of Perron claims that $\beta$ is a solution of Equation~\eqref{linear}.

\begin{claim}\label{claim2}
The function $\beta$ is a solution of \eqref{linear} in $D\setminus \overline{B_y}$.
\end{claim}

\begin{proof}[Proof of Claim~\ref{claim2}]
 Let $p\in D\setminus B_y$ and $\Delta\subset D\setminus
B_y$ a geodesic disk of small radius centered at $p$ as above. By definition of
$\beta$ there exists a sequence of subsolutions $(w_n)$ such that
$w_n(p)\longrightarrow \beta(p)$. Then, consider the sequence of subsolutions
$M_{\Delta}(w_n)$, we have that $M_{\Delta}(w_n)(p)\longrightarrow \beta(p)$.
Also, we have that $M_{\Delta}(w_n)$ is a bounded sequence of solutions of \eqref{linear}
in $\Delta$, so, by considering a subsequence if necessary, we can assume that
it converges to a solution $\overline{w}$ on $\Delta$ with $\beta\geq
\overline{w}$ and $\overline{w}(p)=\beta(p)$. Let us prove that
$\beta=\overline{w}$ on $\Delta$, then $\beta$ will be a solution of
\eqref{linear}. 

We have that $\beta\geq \overline{w}$. Suppose that there is a
point $q\in \Delta$ where $\beta(q)>\overline{w}(q)$. So, there is a subsolution
$s$ such that $s(q)>\overline{w}(q)$. Now consider the sequence of subsolutions
$M_{\Delta}(\max (s,w_n))$. We have that $M_{\Delta}(\max (s,w_n))$ is a
sequence of solutions of \eqref{linear} in $\Delta$. Thus, considering a
subsequence, it converges to a solution $\overline{s}\ge \overline{w}$ of \eqref{linear}
in
$\Delta$.

So, we have $\overline{w}$ and $\overline{s}$ solutions of \eqref{linear} in
$\Delta$, with $\overline{w}(p)=\beta(p)=\overline{s}(p)$, thus by the maximum
principle we have that $\overline{w}=\overline{s}$ in $\Delta$.

But, since $M_{\Delta}(\max (s,w_n))\geq s$, we have that $\overline{s}\geq s$.
This implies that $\overline{s}(q)\geq s(q)>\overline{w}(q)$, which contradicts
$\overline{w}=\overline{s}$ in $\Delta$.
\end{proof}

Until now we have a function $\beta\in C^2(D\setminus\overline{B_y})$ solution
of \eqref{linear} in $D\setminus \overline{B_y}$ such that $u\leq \beta\leq \min
(u+\epsilon', g)$. But we don't have any information about $\beta$ on the boundary
$\partial B_y$. So, the next step is prove that $\beta$ is continuous up to
$\partial B_y$ and $\beta=u+\epsilon'$ in $\partial B_y$.

\begin{claim}\label{claim3}
The function $\beta$ is continuous up to the boundary $\partial B_y$ and takes
the value $u+\epsilon'$ on it.
\end{claim}

\begin{proof}[Proof of Claim~\ref{claim3}] We have that
$\beta(q)=\sup_{w\in\Gamma}w(q)$, then,
$\beta(q)\leq u_{+}(q)\leq u+\epsilon'$. On the other hand,
$u_-(q)=\max(u,v)\in\Gamma$, then $u_-(q)\leq\beta(q)$. Moreover, in $\partial
B_y$ we have that $u_-(q)=u+\epsilon'$. Thus, let $p\in\partial B_y$, and
$\{x_n\}$ a sequence such that $x_n\longrightarrow p$. We have $$u_-(x_n)\leq
\beta(x_n)\leq u_+(x_n),$$ since $$\lim_{x_n\longrightarrow p}
u_-(x_n)=\lim_{x_n\longrightarrow p} u_+(x_n)=u(p)+\epsilon',$$ we have that
$$\lim_{x_n\longrightarrow p} \beta(x_n)=u(p)+\epsilon'.$$

Then we define $\beta(p)$ to be continuous on $D\setminus B_y$ and
$\beta=u+\epsilon'$ in $\partial B_y$.
\end{proof}
This concludes the proof of Proposition~\ref{prop1}.\end{proof}

We have proved the existence of a function $\beta$ defined on $D\setminus
B_y$ such that $\max (u,v)\leq\beta\leq \min (u+\epsilon',g)$ and $\beta\in
C^2(D\setminus\overline{B_y})\cap C^0(D\setminus B_y)$. Now we observe that,
since $v\leq\beta\leq u+\epsilon'$ we have by Theorem 1.1 in \cite{Spruck} that
$\nabla\beta$ is uniformly bounded close to $\partial B_y$. This implies that $\beta$ is a
solution of a
linear uniformly elliptic equation. It follows by Corollary 8.36 in \cite{GT}
that $\beta\in C^{1,\alpha}(D\setminus \overline{B_y})$, and then $\beta\in
C^1(D\setminus \overline{B_y})$.

In the next result we will prove a uniqueness result for CMC graphs in
$\mathbb{H}^2\times\mathbb{R}$ defined over unbounded domain in $\mathbb{H}^2$
whose the existence was proved by A. Folha and S. Melo in \cite{FolhaMelo}.

\begin{proposition}\label{prop2} Let $\beta\in C^1(D\setminus \overline{B_y})$
solution of the Dirichlet problem (\ref{eq1})-(\ref{eq2}) such that
$u\leq\beta\leq u+\epsilon'$. Then, $\beta=u+\epsilon'$.
\end{proposition}

\begin{proof} Let us first analyze the boundary $\partial D$. As in section 2,
let $A_1,B_1,...,A_k,B_k$ be the edges of the $\partial
D$ with $u(A_i)=+\infty=\beta(A_i)$ and $u(B_i)=-\infty= \beta(B_i)$.

For each $i$, let $\mathcal{H}_i(n)$ be a horodisk asymptotic to the vertex
$d_i$ of $D$ such that $\mathcal{H}_i(n+1)\subset \mathcal{H}_i(n)$ and
$\cap_n\mathcal{H}_i(n)=\emptyset$. For each side
$A_i$, let us denote by $A_i^n$ the compact subarc of $A_i$ which is the part of
$A_i$ outside the two horodisks, and by $|A_i^n|$ the length of $A_i^n$.
Analogously, we define $B^n_i$ for each side $B_i$. Denote by $C^n_i$ the
compact arc of $\partial\mathcal{H}_i(n)$ contained in the domain $D$ and let $P^n$
be the subdomain of $D$ bounded by the closed curve formed by the arcs $A^n_i$, $B^n_i$ and
$C^n_i$ and let us denote $\Delta^n=\partial(P^n\setminus \overline{B_y})$.

We have by the theorem of Stokes that
$$
0=\int_{P^n\setminus B_y}Div\left(\frac{\nabla
u}{W_u}-\frac{\nabla\beta}{W_\beta}\right)\\
=\int_{\Delta^n}\left(\frac{\nabla
u}{W_u}-\frac{\nabla\beta}{W_\beta}\right)\cdot\eta
$$
where $W_u^2=1+|\nabla u|^2$ and $W_\beta^2=1+|\nabla\beta|^2$.

Thus

\begin{align*}
0=\int_{A_i^n}\left(\frac{\nabla
u}{W_u}-\frac{\nabla\beta}{W_\beta}\right)\cdot\eta&+\int_{B_i^n}\left(\frac{\nabla
u}{W_u}-\frac{\nabla\beta}{W_\beta}\right)\cdot\eta\\
&+\int_{C_i^n}\left(\frac{\nabla
u}{W_u}-\frac{\nabla\beta}{W_\beta}\right)\cdot\eta+\int_{\partial
B_y}\left(\frac{\nabla u}{W_u}-\frac{\nabla\beta}{W_\beta}\right)\cdot\eta.
\end{align*}

By theorem 5.1 in \cite{FolhaMelo}, these integrals can be estimated. We have
$$
0< |A_i^n|-|A_i^n|+|B_i^n|-|B_i^n|+2|C_i^n|+\int_{\partial
B_y}\left(\frac{\nabla u}{W_u}-\frac{\nabla\beta}{W_\beta}\right)\cdot\eta.
$$

Then
$$
\int_{\partial B_y}\left(\frac{\nabla
u}{W_u}-\frac{\nabla\beta}{W_\beta}\right)\cdot\eta>-2|C_i^n|.
$$

Now, we have that $\beta\leq u+\epsilon'$, suppose that $\beta< u+\epsilon'$, it
follows by Lemmas 1 and 2 in \cite{collin} that $\int_{\partial
B_y}\left(\frac{\nabla u}{W_u}-\frac{\nabla\beta}{W_\beta}\right)\cdot\eta< 0$.
This yields a contradiction, since $|C_i^n|\longrightarrow 0$ when
$n\longrightarrow \infty$, so we conclude that $\beta=u+\epsilon'$.
\end{proof}

Now we are able to prove our main theorem.
\begin{proof}[Proof of Theorem~\ref{main}] We know that $S$ is a properly
immersed CMC surface
contained in $D\times\mathbb{R}$ above $\Sigma$. Then, let $y\in D$, $B_y\subset
D$ and $\epsilon'$ as above. We have three cases to analyze

1) Suppose that there exists $p\in D$ such that $g(p)=u(p)$ (is this the case
we had let aside before Proposition~\ref{prop1}). In this case, by the maximum principle,
we conclude that $u=\beta$.

2) Suppose that $g> u$ and $\inf(g-u)=0$. In this case, by the Proposition
\ref{prop1} there exists $\beta$ solution of (\ref{eq1})-(\ref{eq2}) defined
over $D\setminus B_y$ such that $u\leq\beta\leq g$. Moreover, Proposition
\ref{prop2} assures that $\beta=u+\epsilon'$. This yields a contradiction, since
we assume	 that $\inf (g-u)=0$.

3) Finally, suppose that $g>u$ and $\inf (g-u)=\alpha>0$. Then, pushing up
$\Sigma$ by a vertical translations, that is, by considering $u+\alpha$ instead
of $u$, we have now that $g\geq u+\alpha$ and $\inf (g-\alpha-u)=0$, this case
reduces to cases (1) and (2) and we conclude that $\beta=u+\alpha$, where
$\alpha$ is a constant.
\end{proof}


\begin{thebibliography}{99}

\bibitem{collin} P. Collin and R. Krust. \textit{Le probl\`eme de {D}irichlet
pour l'\'equation des surfaces minimales sur des domaines non born\'es}. Bull.
Soc. Math. France, v. 119, p. 443-462, 1991.

\bibitem{CollRose} P. Collin and H. Rosenberg. \textit{Construction of harmonic
diffeomorphisms and minimal graphs}. Ann. of Math. v. 172, n. 3, p. 1879-1906,
2010. 

\bibitem{DaniHaus} B. Daniel and L. Hauswirth. \textit{Half-space theorem, embedded
minimal annuli and minimal graphs in the Heisenberg group}. Proceedings of the
London Mathematical Society, v. 98, n. 2, p. 445-470, 2009.

\bibitem{DaniMeeksRosen} B. Daniel, William H. Meeks, III, and H. Rosenberg.
\textit{Halfspace theorems for minimal surfaces in $Nil_3$ and $Sol_3$}. Journal
of Differential Geometry, v. 88, n. 1, p. 41-59, 2011.

\bibitem{GalvRosen} J. Galvez and H. Rosenberg. \textit{Minimal surfaces and harmonic
diffeomorphisms from the complex plane onto certain Hadamard surfaces}. Amer. J. Math., v.
132, n. 5, p. 1249–1273, 2010. 

\bibitem{FolhaMelo} A. Folha and S. Melo. \textit{The Dirichlet problem for
constant mean curvature graphs in $\mathbb{H}\times\mathbb{R}$ over unbounded
domains}. Pacific journal of mathematics, v. 251, n. 1, p. 37-65, 2011.

\bibitem{GT} D. Gilbarg and N. Trudinger. \textit{Elliptic Partial Differential
Equations of Second Order}. Springer. New York, 2001.

\bibitem{HausRosenSpru2} L. Hauswirth, H. Rosenberg, and J. Spruck.
\textit{Infinite boundary value problems for constant mean curvature graphs in
$\mathbb{H}^2\times\mathbb{R}$ and $\mathbb{S}^2\times\mathbb{R}$}. American
Journal of Mathematics, v. 131, n. 1, p. 195-226, 2009.

\bibitem{HausRosenSpru} L. Hauswirth, H. Rosenberg, and J. Spruck. \textit{On complete
mean curvature $\frac{1}{2}$ surfaces in $\mathbb{H}^2\times\mathbb{R}$}. Comm.
Anal. Geom, v. 16, n. 5, p. 989-1005, 2008.

\bibitem{HoffmanMeeks} D. Hoffman and William H. Meeks, III. \textit{The strong
halfspace theorem for minimal surfaces}. Inventiones mathematicae, v. 101, n. 1,
p. 373-377, 1990.

\bibitem{Mazet} L. Mazet. \textit{A general halfspace theorem for constant mean
curvature surfaces}. American Journal of Mathematics, v. 135, n. 3, p. 801-834,
2013.

\bibitem{Mazet2} L. Mazet. \textit{The half space property for cmc $1/2$ graphs in
$\mathbb{E}(-1,\tau)$}. preprint.

\bibitem{Menezes} A. Menezes. \textit{A half-space theorem for ideal Scherk
graphs in $M\times\mathbb{R}$}. arXiv preprint arXiv:1306.6305, 2013.

\bibitem{Nelli1} B. Nelli and H. Rosenberg. \textit{Minimal surfaces in
$\mathbb{H}^2\times\mathbb{R}$}. Bull. Braz. Math. Soc. (N.S.), 33(2), 263-292,
2002.

\bibitem{Nelli2} B. Nelli and H. Rosenberg. \textit{Errata. "Minimal surfaces in
$\mathbb{H}^2\times\mathbb{R}$"}. [Bull. Braz. Math. Soc. (N.S.), 33(2002), no
2, 263-292] Bull. Braz. Math. Soc. (N.S.), 38(4), 661-664, 2007.

\bibitem{RodriRosen} L. Rodriguez and H. Rosenberg. \textit{Half-space theorems for mean
curvature one surfaces in hyperbolic space}. Proceedings of the American
Mathematical Society, v. 126, n. 9, p. 2755-2762, 1998.

\bibitem{RosenSchulzSpruc} H. Rosenberg, F. Schulze and J. Spruck. \textit{The
half-space property and entire positive minimal graphs in {$M\times\mathbb{R}$}}. J.
Differential Geom., v. 95, n. 2, p. 321-336, 2013.

\bibitem{Spruck} J. Spruck. \textit{Interior gradient estimates and existence
theorems for constant mean curvature graphs in $\mathbb{H}^n\times\mathbb{R}$}.
Pure Appl. Math Q, v. 3, n. 3, p. 785-800, 2007.


\end{thebibliography}
\end{document}